\documentclass{amsart}
\usepackage{amsmath,amssymb,amsthm,amsfonts,amscd,mathrsfs}
\usepackage{tikz,braids}
\usepackage[all]{xy}

\theoremstyle{plain}
\newtheorem{introtheorem}{Theorem}
\newtheorem{theorem}{Theorem}[section]
\newtheorem{proposition}[theorem]{Proposition}
\newtheorem{lemma}[theorem]{Lemma}
\newtheorem{corollary}[theorem]{Corollary}
\newtheorem{conjecture}[theorem]{Conjecture}

\newtheorem{question}[theorem]{Question}
\newtheorem*{proposition*}{Proposition}
\newtheorem{introcorollary}{Corollary}

\theoremstyle{definition}

\newtheorem{example}[theorem]{Example}
\newtheorem{examples}[theorem]{Examples}

\theoremstyle{remark}
\newtheorem{remark}[theorem]{Remark}

\newcommand{\secref}[1]{Section~\ref{#1}}
\newcommand{\thmref}[1]{Theorem~\ref{#1}}
\newcommand{\propref}[1]{Proposition~\ref{#1}}
\newcommand{\lemref}[1]{Lemma~\ref{#1}}
\newcommand{\corref}[1]{Corollary~\ref{#1}}

\newcommand{\exref}[1]{Example~\ref{#1}}
\newcommand{\queref}[1]{Question~\ref{#1}}

\def\max{\mathrm{max}}

\def\map{\mathrm{map}}

\def\Hom{\mathrm{Hom}}

\def\cat{\mathrm{cat}}
\def\secat{\mathrm{secat}}

\def\dim{\mathrm{dim}}

\def\im{\mathrm{im}}

\def\rk{\mathrm{rank}}
\def\eQ{\mathrm{e}_0}

\newcommand{\be}{\begin{enumerate}}
\newcommand{\ee}{\end{enumerate}}

\newcommand{\Z}{\mathbb{Z}}
\newcommand{\C}{\mathbb{C}}
\newcommand{\Q}{\mathbb{Q}}

\newcommand{\TC}{{\sf TC}}

\newcommand{\MTC}{{\sf MTC}}

\begin{document}

\title[A Mapping Theorem for Topological Complexity]{A Mapping Theorem for Topological Complexity}

\author{Mark Grant}
\author{Gregory Lupton}
\author{John Oprea}

\address{School of Mathematics \& Statistics, Herschel Building, Newcastle University, Newcastle upon Tyne NE1 7RU, U.K.}

\email{mark.grant@newcastle.ac.uk}

\address{Department of Mathematics, Cleveland State University, Cleveland OH 44115 U.S.A.}

\email{g.lupton@csuohio.edu}
\email{j.oprea@csuohio.edu}

\date{\today}


\begin{abstract} We give new lower bounds for the (higher) topological complexity of a space, in terms of the Lusternik-Schnirelmann category of a certain auxiliary space.   We also give new lower bounds for the rational topological complexity of a space, and more generally for the rational sectional category of a map, in terms of the rational category of a certain auxiliary space.  We use our results to deduce consequences for the global (rational) homotopy structure of simply connected, hyperbolic finite complexes.  
\end{abstract}

\thanks{This work was partially supported by grants from the Simons Foundation: (\#209575 to Gregory Lupton
and \#244393 to John Oprea).}

\maketitle
\section{Introduction}\label{sec:intro}

The \emph{topological complexity} of a space is a numerical homotopy invariant, of Lusternik-Schnirelmann type, introduced by Farber \cite{Far03} and motivated by the motion planning problem from the field of topological robotics.  
In this paper, we prove several results that establish new lower bounds for (higher) topological complexity in both the rational and the integral (ordinary) settings.
 We begin with an outline of our main results.  
 
Here,   $\cat(X)$ denotes the \emph{Lusternik-Schnirelmann category} of a space $X$, and
$\secat(p)$ denotes  the \emph{sectional category} of a fibration $p \colon E \to B$.  We refer to \cite{CLOT03} for a general introduction to L-S category and related topics, including sectional category.   Also, $\TC(X)$ denotes the topological complexity of  $X$ and, for $n \geq 3$,  $\TC_n(X)$ denotes the \emph{higher topological complexity} of $X$.
The original articles that introduce these notions \cite{Far03, Rud10}  discuss the connection to the motion planning problem.  We adopt the notational convention that $\TC(X) = \TC_2(X)$. 

Our first main result is as follows (\thmref{thm: cat Y_1 Y_2}).

\begin{introtheorem}\label{introthm: cat Y_1 Y_2}
Suppose given two maps $f_j \colon Y_j \to X$ of connected spaces, $j = 1,2$, with each $(f_j)_\# \colon \pi_i(Y_j) \to \pi_i(X)$ an inclusion  and such that the image subgroups $(f_j)_\# \big(\pi_i(Y_j)\big)$ are complementary subgroups of $\pi_i(X)$, $\forall i \geq 1$ (see \secref{sec: TC integral} for details).  Then for $n \geq 2$ we have 
$$ \cat(Y_1 \times Y_2 \times X^{n-2})  \leq \TC_n(X).$$
\end{introtheorem}
This leads immediately to the following inequality (included in \corref{cor: cat F B}). 
\begin{introcorollary}\label{introcor: cat F B}
Suppose given a fibration sequence $F \to E \to B$ of connected CW complexes.  If the fibration admits a (homotopy) section, then we have
$$\cat(F \times B) \leq \TC(E).$$
\end{introcorollary}

As a consequence, we have the following observation. Suppose that  $X$ dominates $Y$, that is, we have a map $\sigma\colon Y \to X$ that admits a left homotopy inverse $r\colon X \to Y$.  Then it is well-known that $\TC(Y) \leq \TC(X)$ and, since $\cat(Y) \leq \TC(Y)$ in general (see below), we obtain standard lower bounds  $\cat(Y) \leq \TC(Y) \leq \TC(X)$ for $\TC(X)$.  Now if we let $F$ be the homotopy fibre of  $r\colon X \to Y$, then the fibre sequence $F \to X \to Y$ has $\sigma\colon Y \to X$ as a section, and \corref{introcor: cat F B} gives a different lower bound $\cat(F \times Y) \leq \TC(X)$ which, in some cases, may improve upon the standard one.  

We give a number of other applications and consequences of our theorem.  For instance, our theorem leads directly to the following inequality due to Dranishnikov \cite[Th.3.6]{Dra}.

\begin{introcorollary}\label{introcor: wedge}
Suppose $X$ and $Y$ have the homotopy type of connected CW complexes.  Then
$\cat(X \times Y) \leq \TC(X \vee Y)$. 
\end{introcorollary}
\noindent{}Furthermore, in our actual result (\corref{cor: TC X v Y}) we extend this inequality to one involving $\TC_n(X \vee Y)$.

We also show some results in the rational homotopy setting.  The basic  result here is of a similar nature to our integral result, but allows for considerably greater flexibility in its application.  For our rational results, we assume spaces are simply connected.  We refer to \cite{HMR75} for general results about rationalization of simply connected spaces.  We write $X_\Q$ for the rationalization of a simply connected space $X$, and 
$g_\Q\colon X_\Q \to Y_\Q$ for the rationalization of a map $g\colon X \to Y$ of simply connected spaces.  Then $\pi_i(X_\Q) \cong \pi_i(X)\otimes \Q$ are the rational homotopy groups of $X$ and $g_{\Q\#}= g_\#\otimes 1 \colon \pi_*(X_\Q) \to \pi_*(Y_\Q)$ is the homomorphism induced on rational homotopy groups.  In \thmref{thm: rational cat Y_1 Y_2}, we show the following result.  Recall that in several of our results, including this one, we use $\TC_2(X)$ to denote $\TC(X)$.  

\begin{introtheorem}\label{introthm: rational cat Y_1 Y_2}
Suppose given two maps $f_{j} \colon Y_{j} \to X$ of simply connected spaces, $j = 1,2$, such that each $(f_{j\Q})_\# \colon \pi_*(Y_{j\Q}) \to \pi_*(X_\Q)$ is an inclusion on all rational homotopy groups, and the image subgroups satisfy
$$(f_{1\Q})_\# \big(\pi_i(Y_{1\Q})\big) \cap (f_{2\Q})_\# \big(\pi_i(Y_{2\Q})\big) = 0 \subseteq \pi_i(X_\Q),$$
for each $i \geq 2$.  Then for $n \geq 2$ we have
$$ \cat(Y_{1\Q}) +  \cat(Y_{2\Q}) + (n-2)\,\cat(X_\Q)  \leq \TC_n(X_\Q).$$
\end{introtheorem}

The proof we give for \thmref{introthm: rational cat Y_1 Y_2} leads to 
the following consequence (\corref{cor: mapping th}).

\begin{introcorollary}\label{introcor: secat mapping}
Let $g \colon X \to Z$ and $f \colon Y \to Z$ be maps of simply connected spaces.  If
\begin{itemize}\item[(I)] $g_{\Q\#} \colon \pi_*(X_\Q) \to \pi_*(Z_\Q)$ and $f_{\Q\#} \colon \pi_*(Y_\Q) \to \pi_*(Z_\Q)$  are both injective; and
\item[(II)] $\im (g_{\Q\#}) \cap \im(f_{\Q\#}) = 0 \subseteq \pi_*(Z_\Q)$,
\end{itemize}
then we have
$\cat(Y_\Q)  \leq \secat( g_\Q) $.
\end{introcorollary}

\corref{introcor: secat mapping} is worth highlighting, as it is a direct generalization of the so-called \emph{mapping theorem} of F{\'e}lix-Halperin (see \cite[Th.28.6]{F-H-T}), a result of fundamental importance in rational homotopy.  Since we often refer to this result, we will give its statement here, for convenience.

\begin{introtheorem}[{Mapping Theorem [F{\'e}lix-Halperin]}]\label{introthm: mapping th}
Let $f \colon Y \to Z$ be a map of simply connected spaces.  If $f_{\Q\#} \colon \pi_*(Y_\Q) \to \pi_*(Z_\Q)$ is injective, then $\cat(Y_\Q)  \leq \cat(Z_\Q)$. 
\end{introtheorem}

We retrieve the mapping theorem from \corref{introcor: secat mapping} by taking $X = *$, whereupon we have $\secat(g_\Q) = \cat(Z_\Q)$; (I) reduces to the condition that $f_\Q{_\#}$ be injective and (II) becomes vacuous; and the conclusion then reads $\cat(Y_\Q) \leq \cat(Z_\Q)$.  The argument that we give for \corref{introcor: secat mapping}---essentially included in the proof of \thmref{introthm: rational cat Y_1 Y_2}---is recognizably a direct generalization of  the proof of the mapping theorem, too.  This connection accounts for  the title of the paper.
 
In the two previous  results, the rationalized invariants $\secat(X_\Q)$, $\cat(X_\Q)$, and $\TC(X_\Q)$ or $\TC_n(X_\Q)$ each give a lower bound for their corresponding non-rational invariants 
$\secat(X)$, $\cat(X)$, and $\TC(X)$ or $\TC_n(X)$ respectively.  Thus, lower bounds for the rationalized invariants also provide lower bounds for the non-rational invariants. 
We emphasize that, although these results are in rational homotopy, nonetheless they have consequences for ordinary (non-rational) topological complexity.   In \exref{ex: TC(XvX) = 6}, we illustrate how, under the right conditions, our rational results may be used to determine the value of  $\TC(-)$ in a concrete case.

Our rational results also lead to interesting consequences for the global rational homotopy structure of finite complexes.  We mention two such here, to suggest the kinds of conclusions we are able to draw.  For $X$ a simply connected, finite complex, we say that $X$ is \emph{(rationally) hyperbolic} if it has infinitely many non-zero rational homotopy groups.  In some sense, this is the ``generic" behaviour of a finite complex (see \cite[Part VI]{F-H-T}).   If a simply connected space $X$ has $\TC(X) = 2$, then $\TC(X_\Q) \leq 2$, and also $\cat(X_\Q) \leq 2$.  The interest in the following result lies in the case in which $\cat(X_\Q) = 2$. It is included as part of \corref{cor: TC 2 A-F}.

\begin{introtheorem}\label{introthm: TC A-F}
Let $X$ be a simply connected,  hyperbolic finite complex.
If $\TC(X_\Q) = 2$, then $X$ has some connective cover that is a rational co-H-space.  In particular, $\pi_*(\Omega X) \otimes \Q$ contains a free Lie algebra on two generators.
\end{introtheorem}

Now a well-known, long-standing open conjecture in rational homotopy is that the rational homotopy Lie algebra $\pi_*(\Omega X) \otimes \Q$ of a simply connected,  finite complex that is hyperbolic should have a sub-Lie algebra that is free on two generators.   This is the so-called \emph{Avramov-F{\'e}lix conjecture}, and it appears as Problem 4 in \cite[Sec.39]{F-H-T}.  If true, the conjecture would help explain the phenomenon of exponential growth of rational homotopy groups, about which a great deal of work has been done.
The conjecture has been established for $X$ with $\cat(X_\Q) =  2$ \cite{F-H-T84, Fe-Th86}.  However, our conclusion here is somewhat stronger and  follows by a simple argument, once given our basic result.  With our methods, we are able to glean various conclusions in this direction, such as the following new cases of the Avramov-F{\'e}lix conjecture (included in \corref{cor: TC 3 A-F}).  Recall again that $\TC_2(X) = \TC(X)$.

\begin{introcorollary}\label{introcor: TC 3 A-F}
Let $X$ be a simply connected, hyperbolic finite complex with $\cat(X_\Q)$ $= 3$ and $\TC_n(X_\Q) = 3n-3$ for some $n\geq 3$.  Then $\pi_*(\Omega X) \otimes \Q$ contains a free Lie algebra on two generators.
\end{introcorollary}

To complete this outline of our results, we mention that \thmref{introthm: cat Y_1 Y_2} may be specialized  to the main result of  \cite{GLO13a}, which may be viewed as a companion article to this one.  In \cite{GLO13a}, we establish a new lower bound for $\TC(X)$ with $X = K(G, 1)$ an aspherical space, in terms of the cohomological dimension of a certain auxiliary subgroup of $G \times G$.  Our  proof of the result there is rather different from our proof of  \thmref{introthm: cat Y_1 Y_2}, and it involves the so-called \emph{one-dimensional category}, denoted $\cat_1(-)$, of the auxiliary group.   The applications of our result in \cite{GLO13a}  complement those we obtain here, as they concern aspherical spaces.

The paper is organized as follows.  In \secref{sec: TC integral} we work in ordinary homotopy theory.  \thmref{introthm: cat Y_1 Y_2} of this introduction appears there as \thmref{thm: cat Y_1 Y_2}.    This section also contains \corref{introcor: cat F B},  \corref{introcor: wedge}, and several concrete applications of these results.   \secref{sec: rational} contains our main rational result which is  \thmref{introthm: rational cat Y_1 Y_2} of this introduction.  In \exref{ex:MTC example}, we compare our result with the approach of \cite{F-G-K-V06}, and in \corref{cor: rational TC = 1} we use our result to analyze the cases in which $\TC(X_\Q) = 1$ or $\TC_n(X_\Q) = n-1$ for $n \geq 3$.  We end this section with \corref{cor: zero bracket}, a further application of \thmref{introthm: rational cat Y_1 Y_2} in which we establish a connection between  $\TC(X_\Q)$ (or $\TC_n(X_\Q)$) and the bracket structure in $\pi_*(\Omega X) \otimes \Q$. 
In \secref{sec: rational to integral}, we briefly illustrate some situations in which our rational results may be levered to yield calculations of (integral) $\TC(X)$ for certain spaces $X$.  In the final \secref{sec:A-F}, we apply \thmref{introthm: rational cat Y_1 Y_2} to analyze the behavior of $\TC(-)$ (or $\TC_n(-)$) with respect to connective covers, and draw our conclusions concerning the Avramov-F{\'e}lix conjecture.

We finish this introduction with a brief r{\'e}sum{\'e} of definitions and basic facts.  
Recall that  $\cat(X)$---the Lusternik-Schnirelmann category of $X$---is the smallest $n$ for which there is an open covering $\{ U_0, \ldots, U_n \}$ by $(n+1)$ open sets, each of which is contractible in $X$.  
The \emph{sectional category} of a fibration $p \colon E \to B$, denoted by $\secat(p)$, is the smallest number $n$ for which there is an open covering $\{ U_0, \ldots, U_n \}$ of $B$ by $(n+1)$ open sets, for each of which there is a local section $s_i \colon U_i \to E$ of  $p$, so that $p\circ s_i = j_i \colon U_i \to B$, where $j_i$ denotes the inclusion.  
We refer to \cite{CLOT03} for a general introduction to L-S category and related topics, such as sectional category. 
 Let $PX$ denote the space of (free) paths on a space $X$.  There is a fibration $P_2\colon PX \to X\times X$, which evaluates a path at initial and final point: for $\alpha \in PX$, we have $P_2(\alpha) = \big(\alpha(0), \alpha(1)\big)$.  This is a fibrational substitute for the diagonal map $\Delta \colon X \to X \times X$.  We define the \emph{topological complexity} $\TC(X)$ of $X$ to be the sectional category $\secat\big( P_2\big)$ of this fibration.  

We also consider the ``higher analogues"  of topological complexity introduced by Rudyak in \cite{Rud10} (see also \cite{Rud10b} and \cite{BGRT10}).  This notion may also be motivated by a motion planning problem of a constrained type.
Let $n \geq 2$ and consider the fibration
\[P_n \colon PX \to X \times \cdots \times X = X^n,\]
defined by dividing the unit interval $I = [0, 1]$ into $(n-1)$ subintervals of equal length, with $n$ subdivision points $t_0 = 0, t_1 = 1/(n-1), \ldots, t_{n-1} = 1$ (thus $(n-2)$ subdivision points interior to the interval), and then evaluating at each of the $n$ subdivision points:
\[P_n(\alpha) = \big(  \alpha(0), \alpha(t_1), \ldots, \alpha(t_{n-2}), \alpha(1)\big),\]
for $\alpha \in PX$.  This is a fibrational substitute for the $n$-fold diagonal $\Delta_n\colon X \to X^n$.   Then the \emph{higher topological complexity} $\TC_n(X)$ is defined as $\TC_n(X) = \secat(P_n)$.  

For the connection to the motion planning problem, we refer to the original articles  \cite{Far03, Rud10}.  Note that these authors use un-normalized $\TC_n(-)$, which is one more than our (normalized) $\TC_n(-)$.  Most of our results apply to ordinary $\TC(X)$ as well as higher $\TC_n(X)$, for $n \geq 3$.  

There are several well-known and basic facts which we use throughout the article to compare sectional categories of various maps.  We state them here; each is easily justified directly from the definitions. 
First, suppose given a fibration $p \colon E \to B$ and any map $f \colon B' \to B$.  Form the  pullback $p' \colon E' \to B'$ of $p$ along $f$.  
Then we have $\secat(p') \leq \secat(p)$.   Next, if $E$ is contractible or, more generally, if $p\colon E \to B$ is nulhomotopic,  then we have $\secat(p) = \cat(B)$.  Combining these ingredients readily leads to general inequalities  
$$\cat(X^{n-1}) \leq \TC_{n}(X) \leq \cat(X^{n}) $$
for each $n \geq 2$.   These basic facts are explained in  \cite{BGRT10, LuSc13}, amongst other places.

We frequently use $n$-connective covers of a space.  Our notation for this is $X^{[n]}$.  This is the homotopy fibre of the map from $X$ to its $n$th Postnikov section, and it is a connected space that satisfies $\pi_i(X^{[n]}) = 0$ for $i \leq n$, and  $\pi_i(X^{[n]}) = \pi_i(X)$ for $i \geq n+1$.

\section{A lower bound on (higher) topological complexity}\label{sec: TC integral}

From now on, suppose that spaces $X$ and $Y$ are connected and of the homotopy type of a CW complex.  Constructions from them, such as $\Omega X$ or  pullbacks that involve them,  may be disconnected.
We begin with a result that may be well-known, but which we cannot find in the literature in the form that we need.     We include a proof here, for completeness.  The case $n=2$ is proved in  \cite[Prop.3.2]{McC90} (see also the proof of \cite[Th.7]{Co-Fa10}).

In the following Proposition,  $\mu\colon \Omega X \times \Omega X \to \Omega X$ and $\iota\colon \Omega X \to \Omega X$ denote the usual loop multiplication map and  the inverse map of loops.  Also, $\pi_{j, j+1} = (\pi_j, \pi_{j+1}) \colon (\Omega X)^n \to (\Omega X) \times (\Omega X)$, for $j = 1, \ldots, n-1$, denotes the projection onto the two consecutive $j$th and $(j+1)$st factors of the product.

\begin{proposition}\label{prop: connecting hom}
For $n \geq 2$, consider  the fibration sequence 
$$\xymatrix{(\Omega X)^{n-1} \ar[r] & PX \ar[r]^{P_n} &  X^n}$$
of the fibration $P_n$ used in the above definition of $\TC_n(X)$.   Then the connecting map of this fibration sequence may be identified as a map
$$\partial = (\partial_1, \ldots, \partial_{n-1}) \colon  \Omega(X^n) = (\Omega X)^n \to  (\Omega X)^{n-1},$$
with coordinate functions $\partial_j = \mu\circ (\iota \times 1)\circ \pi_{j, j+1}$ 
for each $j = 1, \dots, n-1$.
On homotopy groups, we may identify the induced homomorphism  $\partial_\# \colon \pi_r(\Omega X^n) \to \pi_r(\Omega X^{n-1})$, for $r \geq 1$, as 
a homomorphism of abelian groups
$$\partial_\# \colon \bigoplus_{j=1}^{n} \pi_{r+1}(X) \to \bigoplus_{j=1}^{n-1} \pi_{r+1}(X),$$
where, for $\mathbf{a} = (a_1, \dots, a_n) \in \pi_r(\Omega X^n) = \oplus_{j=1}^{n} \pi_{r+1}(X)$,  we have
$$\partial_\#(\mathbf{a}) = (- a_1 + a_2, - a_2 + a_3, \dots, - a_{n-1} + a_n).$$
On  sets of components,  the induced map of based sets $\partial_\# \colon \pi_0(\Omega X^n)$ $\to \pi_0(\Omega X^{n-1})$   may be  identified as 
$$\partial_\# \colon \prod_{j=1}^{n} \pi_1(X) \to \prod_{j=1}^{n-1} \pi_1(X),$$
with $\partial_\#(\mathbf{a}) = ((a_1)^{-1}a_2, (a_2)^{-1} a_3, \dots, (a_{n-1})^{-1}a_n)$.
\end{proposition}

\begin{proof}
Recall that, for any based map of based spaces $f \colon X \to Y$, we obtain the homotopy fibre $\mathbf{T}^f$ as a pullback of the path fibration along $f$, thus
$$\xymatrix{ \mathbf{T}^f \ar[r] \ar[d] & \mathcal{P} Y \ar[d]^{p_0} \\
X \ar[r]_{f} & Y.}$$
Here, $\mathcal{P} Y$ denotes the based path space $\mathcal{P} Y = \{ \gamma \colon I \to Y \mid \gamma(1) = y_0 \}$ and $p_0(\gamma) = \gamma(0)$.  Then, if $f \colon X \to Y$ is itself a fibration, with fibre $F = f^{-1}(y_0)$, we have a homotopy equivalence $\beta \colon F \to  \mathbf{T}^f$, which maps $x \mapsto (x, C_{y_0}) \in \mathbf{T}^f$.  Also, we have a whisker map $\phi\colon \Omega Y \to \mathbf{T}^f$ given by $\phi(\zeta) = (x_0, \zeta)$ for each based loop $\zeta$.  The connecting map $\partial \colon \Omega Y \to F$ of the fibration is then given by $\partial = \beta^{-1} \circ \phi$, where $\beta^{-1}$ is a homotopy inverse for $\beta$.  This is summarized in the following diagram:
$$\xymatrix{ & F \ar[rd]^{\mathrm{incl.}} \ar[d]^{\beta}_{\simeq} \\
\Omega Y \ar[r]_{\phi} \ar[ru]^{\partial} &   \mathbf{T}^f \ar[r] & X \ar[r]^{f} & Y.}$$

Now consider the fibration $P_n \colon PX \to X^n$ as above.  The fibre here is $F = (\Omega X)^{n-1}$, and the homotopy fibre is
$$ \mathbf{T}^{P_n} = \{ (\alpha, \gamma) \in PX \times \mathcal{P}X^n \mid \alpha(t_i) = \gamma_{i+1}(0), i=0,  \ldots, n-1 \},$$
where $\{t_i\}$ is the subdivision of $[0,1]$ with $t_i = i/(n-1)$, as in the introduction, and we  have written the coordinate functions of $\gamma$ as $\gamma = (\gamma_1, \ldots, \gamma_n)\colon I \to X^n$.  The homotopy equivalence $\beta\colon (\Omega X)^{n-1} \to \mathbf{T}^{P_n}$ from fibre to homotopy fibre is given by $\beta(\eta) = (\eta, C_{\mathbf{x_0}})$, where $C_x$ denotes the constant path at $x$, and $\mathbf{x_0} \in X^n$ denotes the base point  $\mathbf{x_0} = (x_0, \dots, x_0)$.  
Now for $\alpha \in PX$, let $\alpha[i] \colon [t_{i-1}, t_i] \to X$ denote the restriction of $\alpha$ to the subinterval  $[t_{i-1}, t_i]$, for $i = 1, \ldots, n-1$.
Then the inverse homotopy equivalence may be written explicitly as $\beta^{-1} \colon \mathbf{T}^{P_n} \to (\Omega X)^{n-1}$, with
$$ \beta^{-1}(\alpha, \gamma) = (\overline{\gamma_1}\alpha[1] \gamma_2, \overline{\gamma_2}\alpha[2] \gamma_3, \ldots, \overline{\gamma_{n-1}}\alpha[n-1] \gamma_n).$$
Here, the notation $\overline{\xi}$ denotes the inverse path to $\xi$, i.e., $\overline{\xi}(t) = \xi(1-t)$, and juxtaposition of paths denotes their usual composition. 

From the pullback diagram
$$\xymatrix{\Omega X^n  \ar@/^1pc/[rrd]^{\mathrm{incl.}} \ar@/_1pc/[ddr]_{{*} = C_{x_0}} \ar@{..>}[rd]^{\phi}\\
 & \mathbf{T}^{P_n} \ar[r] \ar[d] & \mathcal{P} X^n \ar[d]^{p_0} \\
 & PX \ar[r]_{P_n} & X^n,}$$
we obtain the whisker map $\phi\colon \Omega X^n \to \mathbf{T}^{P_n}$ as 
$\phi(\gamma) = ( C_{x_0}, \gamma)$, whence we have the connecting map $\partial\colon \Omega X^n \to \Omega X^{n-1}$ as
$$\partial(\gamma) = \beta^{-1}\circ \phi(\gamma) = (\overline{\gamma_1} \gamma_2, \overline{\gamma_2} \gamma_3, \ldots, \overline{\gamma_{n-1}} \gamma_n).$$
In terms of maps, then, we may identity the $j$th component of the connecting map as 
$$\partial = \mu\circ (\iota \times 1)\circ \pi_{j, j+1} \colon \Omega X^n  \to \Omega X \times \Omega X   \to \Omega X \times \Omega X  \to \Omega X,$$
for $j = 1, \dots, n-1$.   The assertions about the homomorphisms induced on homotopy groups and sets of components follow, as $\mu$ induces the usual addition of homotopy elements and $\iota$ the inverse (anti-)homomorphism. 
\end{proof}

Note that, in the above result, whilst $\pi_r(\Omega X^n)$ and $\pi_r(\Omega X^{n-1})$ are both groups for $r \geq 0$, they may be non-abelian for $r=0$, and so the induced map $\partial_\#$ may fail to be a homomorphism for $r=0$. 

We will need the following generalization of the direct product of (non-abelian) groups: say that subgroups $A$ and $B$ of a  group $G$ are \emph{complementary subgroups} if they (1) span $G$, that is, if any element of $G$ may be written as a product $a b$ for some $a \in A$ and $b \in B$, and (2) also intersect trivially, that is, if $A \cap B = \{e\}$.   
Note that, if  $A$ and $B$ are complementary subgroups of an \emph{abelian} group $G$, then $G$ is isomorphic to the direct sum $A\oplus B$.

\begin{theorem}\label{thm: cat Y_1 Y_2}
Suppose given maps $f_j \colon Y_j \to X$ of connected spaces, $j = 1,2$, such that, for each $i \geq 1$, we have:
\begin{itemize}
\item[(i)] each $(f_j)_\# \colon \pi_i(Y_j) \to \pi_i(X)$ is an inclusion;
\item[(ii)]  $(f_1)_\# \big(\pi_i(Y_1)\big)$ and $(f_2)_\# \big(\pi_i(Y_2)\big)$ are complementary subgroups in  $\pi_i(X)$.
\end{itemize}
Then for $n \geq 2$, we have
$$ \cat(Y_1 \times Y_2 \times X^{n-2})  \leq \TC_n(X).$$
In particular, with $n = 2$, we have 
$ \cat(Y_1 \times  Y_2)  \leq \TC(X).$
\end{theorem}

\begin{proof}
Write $Y = Y_1 \times Y_2\times X^{n-2}$, and $f = f_1 \times f_2 \times 1_{X^{n-2}} \colon Y \to X^n$.  Then let $N(f,P_n)$ denote the pullback  
$$\xymatrix{ N(f,P_n)\ar[r] \ar[d]_{\overline{P_n}} & PX \ar[d]^{P_n} \\
Y \ar[r]_{f} & X^n.}$$
Notice that, since $P_n$ is a fibration, this pullback is also a homotopy pullback.
 Extend this commutative diagram to a homotopy commutative ladder of the fibre sequences of each vertical map, thus:
$$\xymatrix{\Omega Y\ar[r]^{\Omega f} \ar[d]_{\delta} & \Omega X^n \ar[d]^{\partial} \\
 \Omega X^{n-1} \ar@{=}[r] \ar[d] &  \Omega X^{n-1} \ar[d]\\
 N(f,P_n)\ar[r] \ar[d]_{\overline{P_n}} & PX \ar[d]^{P_n} \\
Y \ar[r]_{f} & X^n.}$$
In the long exact homotopy sequence of $\overline{P_n}$, identify the connecting map $\delta\colon \Omega Y \to \Omega X^{n-1}$ as $\delta = \partial\circ \Omega f$, where $\partial$ denotes the connecting map from \propref{prop: connecting hom}.  So on homotopy groups (or homotopy sets, for $r = 0$), we may identify the induced map $\delta_\# \colon \pi_r(\Omega Y) \to \pi_r(\Omega X^{n-1})$ as a map
$$\delta_\# = \partial_\# \circ f_\#\colon \pi_{r+1}(Y_1) \times \pi_{r+1}(Y_2) \times \prod_{j=3}^{n} \pi_{r+1}(X) \to \prod_{j=1}^{n-1} \pi_{r+1}(X),$$
for each $r \geq 0$.  With  $b_j \in \pi_{r+1}(Y_j)$, $j = 1, 2$, and $a_j \in \pi_{r+1}(X)$, $j = 3, \dots, n$, we have
$$\delta_\#(b_1, b_2, a_3, \dots, a_n) = \big( - f_{1\#}(b_1) + f_{2\#}(b_2), - f_{2\#}(b_2) + a_3, - a_3 + a_4, \dots, - a_{n-1} + a_n\big),$$
if $r \geq 1$, and 
$$\delta_\#(b_1, b_2, a_3, \dots, a_n) = \big( ( f_{1\#}(b_1))^{-1} f_{2\#}(b_2), (f_{2\#}(b_2))^{-1} a_3, a_3^{-1}a_4, \dots, (a_{n-1})^{-1}a_n\big)$$
if $r= 0$. 
Because the images $f_{1\#} \big(\pi_{r+1}(Y_1)\big)$ and $f_{2\#} \big(\pi_{r+1}(Y_2)\big)$  span $\pi_{r+1}(X)$, for each $r \geq 0$, $\delta_\#$ is surjective. Because these images intersect only in $0$ (or $\{e\}$, in the case $r = 0$), and we are also assuming that $f_{1\#}$ and $f_{2\#}$ are injective, it follows that $\delta_\#$ is injective.  Thus, $\delta_\# \colon \pi_r(\Omega Y) \to \pi_{r}(\Omega X^{n-1})$ is an isomorphism for each $r \geq 1$, and a bijection of based sets for $r=0$.   Therefore,  $N(f,P_n)$ is weakly contractible, and thus contractible, since it is of the homotopy type of a CW complex.

In the first pullback diagram, then, it follows that $\secat(\overline{P_n}) = \cat(Y)$.  Since this diagram is a pullback, we have
$$\cat(Y) = \secat(\overline{P_n})  \leq \secat(P_n) = \TC_n(X). \qed$$  
\renewcommand{\qed}{}\end{proof}

\begin{corollary}[to the proof of \thmref{thm: cat Y_1 Y_2}]\label{cor: cat secat}
For maps $f_j \colon Y_j \to X$  that satisfy the hypotheses of  \thmref{thm: cat Y_1 Y_2}, we have
$\cat(Y_1) \leq \secat(f_2)$. 
\end{corollary}

\begin{proof}
Consider the (standard) homotopy pullback 
$$\xymatrix{ M(f_1,f_2)\ar[r]^-{\overline{f_1}} \ar[d]_{\overline{f_2}} & Y_2 \ar[d]^{f_2} \\
Y_1 \ar[r]_{f_1} & X.}$$
That is, 
$$M(f_1, f_2) = \{ (y_1, \alpha, y_2) \in Y_1 \times \map(I, X) \times Y_2 \mid \alpha(0) = f_1(y_1), \alpha(1) = f_2(y_2) \}.$$
Notice that $M(f_1, f_2) = N(f, P_2)$, where $f = f_1 \times f_2\colon Y_1 \times Y_2 \to X \times X$ and $P_2 \colon PX \to X \times X$ are the maps involved in the proof of \thmref{thm: cat Y_1 Y_2} for $n = 2$, and $N(f, P_2)$ is the pullback.  Under the hypotheses of \thmref{thm: cat Y_1 Y_2}, we showed that $N(f, P_2)$ is contractible.  It follows from the homotopy pullback that $\cat(Y_1) = \secat(\overline{f_2}) \leq \secat(f_2)$.
\end{proof}

There is a natural way in which the maps involved in \thmref{thm: cat Y_1 Y_2} arise.

\begin{corollary}\label{cor: cat F B}
Suppose given a fibration sequence $F \to E \to B$ of connected CW complexes.  If the fibration admits a (homotopy) section, then for each $n \geq 2$, we have
$$\cat(F \times B \times E^{n-2}) \leq \TC_n(E).$$
\end{corollary}

\begin{proof}
The maps $F \to E$ and $B \to E$, with the latter being the section, satisfy the conditions of the theorem, with $F, B, E$ identified with $Y_1, Y_2, X$, respectively. 
\end{proof}

We present several examples of situations covered by \corref{cor: cat F B}.  Here, we use $\widetilde{X}$ to denote the universal cover of $X$.

\begin{corollary}
If $X$ has fundamental group $\Z$, then $\TC(X) \geq \cat(\widetilde{X} \times S^1)$.  If $X$ has fundamental group $\Z \times \Z$, then $\TC(X) \geq \cat(\widetilde{X} \times S^1 \times S^1)$.
\end{corollary}

\begin{proof}
We argue for $\pi_1(X) \cong \Z \times \Z$; the argument for $\pi_1(X) \cong \Z$ is identical.   Let $k \colon X \to S^1\times S^1 = K(\pi_1(X), 1)$ be the classifying map of the universal cover, so that we have a fibre sequence $\widetilde{X} \to X \to S^1 \times S^1$.  We obtain a section of  $k$ by first choosing a map $s\colon S^1 \vee S^1 \to X$ for which $k \circ s \sim J\colon S^1 \vee S^1 \to S^1 \times S^1$, where $J\colon S^1 \vee S^1 \to S^1 \times S^1$ is the inclusion, and noting that $s$ extends to a map $\sigma \colon S^1 \times S^1 \to X$, as $\pi_1(X)$ is abelian.  Then $k\circ \sigma     \colon S^1 \times S^1 \to S^1 \times S^1$ induces an isomorphism on $\pi_1(-)$ and thus, since $S^1\times S^1 = K(\pi_1(X), 1)$, is a homotopy equivalence.  We may adjust $\sigma$, if necessary by pre-composing with a suitable self-homotopy equivalence of $S^1 \times S^1$, into a section of $k$   up to homotopy.  The inequality now follows from \corref{cor: cat F B}. 
\end{proof}

Associated bundles provide another source of potential applications.  Suppose that $p\colon E \to M$ is a principal $G$-bundle, for $G$ a  topological group.  For $X$ a $G$-space, we have the associated bundle $E \times_G X \to M$, with fibre $X$.  
\begin{corollary}
With the above notation, if the action of $G$ on the connected space $X$ has a fixed point, then we have $\TC(E \times_G X) \geq \cat(X \times M)$. 
\end{corollary}

\begin{proof}
If the action of $G$ on $X$ has a fixed point, then  $E \times_G X \to M$ has a section, and we may apply  
\corref{cor: cat F B}.
\end{proof}

\begin{examples}
Here, since the bundle $E \times_G X \to M$ admits a section, we already have the lower bound $\TC(E \times_G X) \geq \TC(M)$.  In this situation, therefore, our new lower bound will be useful when $\TC(M)$ is relatively small, compared with $\cat(M)$.  This happens if, for instance, we also suppose that $M$ is a group, in which case we have $\TC(M) = \cat(M)$ (see \cite[Lem.8.2]{Far04}, \cite{LuSc13}).  For example, take the principal circle bundle $U(n) \to PU(n)$, and let $X = S^2$ with the rotation action of the circle $S^1$.  Here we obtain a lower bound of $\TC(U(n) \times _{S^1} S^2) \geq \cat\big(S^2 \times PU(n)\big)$.  
As another example,  take the Hopf bundle $S^1 \to S^3 \to S^2$, and consider any $S^1$-action on $\C P^n$.  The action must have a fixed-point, as $\chi( \C P^n) \not=0$.  We obtain that $\TC( S^3 \times_{S^1} \C P^n) \geq \cat(\C P^n \times S^2) = n+1 > \cat(S^2)$. 
One advantage of our approach here is that it avoids (spectral sequence) cohomological calculations with the twisted product $E \times_G X$.   \end{examples}

Mapping tori are a further source of applications of \corref{cor: cat F B}.

\begin{corollary}
Suppose $\phi\colon M \to M$ is a diffeomorphism of a connected, smooth manifold $M$, and let $M_\phi$ denote the corresponding mapping torus.  Then we have  $\TC(M_\phi) \geq \cat(M \times S^1)$.  
\end{corollary}

\begin{proof}
\corref{cor: cat F B} obtains the lower bound, for we have a sectioned fibration $M \to M_\phi \to S^1$.  
\end{proof}

Using \corref{cor: cat F B}, we may also retrieve, and in a very transparent way, (part of) a result of Dranishnikov, as well as its extension to higher topological complexity.

\begin{corollary}\label{cor: TC X v Y}
Suppose $X$ and $Y$ have the homotopy type of connected CW complexes.  Then
$$\mathrm{max}\{ \cat(X \times Y^{n-1}), \cat(Y \times X^{n-1}) \} \leq  
\cat(X \times Y \times (X \vee Y)^{n-2}) \leq \TC_n(X \vee Y),$$
for $n \geq 2$.  In particular, we have $\cat(X \times Y) \leq \TC(X \vee Y)$. 
\end{corollary}

\begin{proof}
Let $p_1\colon X \vee Y \to X$ be projection onto the first summand, $i_1 \colon X \to X \vee Y$ inclusion into the first summand, and $i_2 \colon Y \to X \vee Y$ inclusion into the second summand.  Because the composition $p_1\circ i_2 \colon Y \to X$ is null, there is a lift, up to homotopy, of $i_2$ through the (homotopy) fibre inclusion $j \colon F \to X \vee Y$.  Furthermore, because $i_2$ admits the retraction $p_2\colon X \vee Y \to Y$, this lift admits the retraction $p_2 \circ j \colon F \to Y$.  All this is summarized in the following diagram.
$$\xymatrix{ & F \ar[d]^{j}  \ar@/_1pc/[dl]_{p_2 \circ j}\\ 
Y \ar[ur] \ar[r]_-{i_2} & X \vee Y \ar[d]_{p_1}\\
 & X  \ar@/_1pc/[u]_{i_1} }$$
Now \corref{cor: cat F B} gives $\cat(F \times X \times (X \vee Y)^{n-2}) \leq \TC_n(X \vee Y)$.  But $Y$ is a retract of $F$, and hence $Y \times X \times (X \vee Y)^{n-2}$ is a retract of $F \times X \times (X \vee Y)^{n-2}$.  The second inequality follows.  The first follows because both $X^{n-2}$ and $Y^{n-2}$ are retracts of $(X \vee Y)^{n-2}$.     
\end{proof}

\section{Rational Results}\label{sec: rational}

We now turn our attention to the rational homotopy setting.  Here, we consider simply connected spaces.     If $X$ is a simply connected space, we denote by $X_\Q$ its rationalization.  Likewise, we denote by $f_\Q\colon X_\Q \to Y_\Q$ the rationalization of a map $f \colon X \to Y$ of simply connected spaces.  We abuse notation somewhat by using $X_\Q$, respectively, $f_\Q$, to denote a simply connected, rational space, respectively, map between such spaces,  regardless of whether we have a particular de-rationalization $X$, respectively, $f$, to hand.   The rational version of $\TC_n(X)$ that we consider will just be $\TC_n(X_\Q)$, the (higher) topological complexity of the rationalization of $X$.  Generally, this provides a lower bound for the (higher) topological complexity.   This fact is well-known, and easy to see, but we include a formal statement and proof here for the sake of completeness.  We may as well consider $P$-localization for this, as we use only generalities about localization.  The corresponding inequality for $\cat(-)$ is a result of Toomer \cite{Too75}.   

\begin{lemma}\label{lem: local TC}
Let $f \colon E \to B$ be a fibration of simply connected spaces, and $f_P\colon X_P$ $\to Y_P$  be its $P$-localization (at any set of primes $P$).  Then $\secat(f_P) \leq \secat(f)$.  In particular, we have $\TC_n(X_P) = \secat\big( (P_n)_P \big) \leq \secat(P_n) = \TC_n(X)$.
\end{lemma}

\begin{proof}
We may use general properties of localization (see e.g.~\cite[II.1]{HMR75}), although not with the covering definition of $\secat(-)$ that we have given.  Instead, we use Svarc's fibrewise join definition of $\secat(-)$.  For this, one starts with the pullback along $f$ of $f$, and then forms the pushout of the top left corner to obtain the filler in the following diagram:
$$
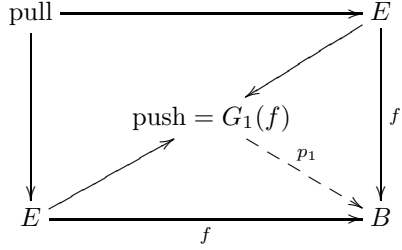
\xymatrix{ \mathrm{pull}\ar[rr] \ar[dd]& & E \ar[dd]^{f} \ar[ld]\\ 
 & \mathrm{push}=G_1(f) \ar@{-->}[rd]^{p_1} &  \\
E\ar[rr]_{f} \ar[ru]& & B }$$
Iterating this ``pullback-pushout" step $n$-times, each time starting with the pullback along $f$ of $p_{i-1}$, results in a fibration $p_n\colon G_n(f) \to Y$, called the $n$-fold fibrewise join (of $f$ with itself).  Under mild conditions (e.g.~$B$ paracompact, or normal), a result of Svarc says that $\secat(f)$ equals the smallest $n$ for which  $p_n\colon G_n(f) \to Y$ has a (global) section.  A proof of this identification is given in \cite[Th.2.2]{F-G-K-V06}.  

Now suppose that $\secat(f) = n$, and that $l_B \colon B \to B_P$ is a $P$-localization map.  The functorial nature of the fibrewise join construction results in a commutative diagram
$$
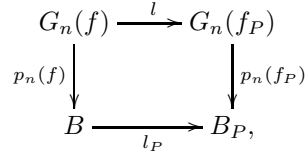
\xymatrix{ G_n(f) \ar[r]^l \ar[d]_{p_n(f)}&  G_n(f_P)  \ar[d]^{p_n(f_P)} \\ 
B\ar[r]_{l_P} & B_P, }$$
 and it is not difficult to identify the right-hand vertical map as equivalent to $(p_n)_P\colon  (G_n(f))_P \to B_P$ (see \cite{Too75} for the case in which $f$ is the path fibration).  Since we assume that $\secat(f) = n$, the left-hand vertical map admits a section, and the universal property of the localization map $l_B \colon B \to B_P$ yields a section of 
 $p_n(f_P) \colon G_n(f_P)  \to B_P$.  That is, we have $\secat(f_P) \leq n$.
 
 The first equality of the last assertion follows from the fact that the induced map  $(l_P)_*\colon \map(I, X) \to \map(I, X_P)$ is a $P$-localization \cite[Th.II.3.11]{HMR75}.
\end{proof}

Note that the simply connected hypothesis is necessary in \lemref{lem: local TC}; the inequality may fail if the spaces are assumed only to be nilpotent.    For instance, if we take $X = S^1$, then we have $\TC(X) = \cat(X) = 1$.  However, $X_\Q = S^1_\Q$ is well-known to have $\cat(S^1_\Q) = 2$.  Since $S^1$ is an H-space, it follows that  $X_\Q = S^1_\Q$ is also an H-space.  Therefore, we have $\TC(X_\Q) = \cat(S^1_\Q) = 2$, and $\TC(X_\Q) > \TC(X)$ here.

We begin with a result in this setting whose conclusion is the same as that of \thmref{thm: cat Y_1 Y_2}, but whose hypotheses are weaker (the images of the maps in rational homotopy groups need not span).  As we shall see, this allows for greater flexibility in applying the result.

\begin{theorem}\label{thm: rational cat Y_1 Y_2}
Suppose given maps $f_{j\Q} \colon Y_{j\Q} \to X_\Q$ of simply connected, rational  spaces, $j = 1,2$, such that each $(f_{j\Q})_\# \colon \pi_*(Y_{j\Q}) \to \pi_*(X_\Q)$ is an inclusion on all (rational) homotopy groups, and the image subgroups satisfy
$$(f_{1\Q})_\# \big(\pi_i(Y_{1\Q})\big) \cap (f_{2\Q})_\# \big(\pi_i(Y_{2\Q})\big) = 0 \subseteq \pi_i(X_\Q),$$
for each $i \geq 2$.  Then for $n \geq 2$, we have
$$ \cat(Y_{1\Q}) +  \cat(Y_{2\Q}) + (n-2)\,\cat(X_\Q)  \leq \TC_n(X_\Q).$$
In particular, with $n = 2$, we have 
$ \cat(Y_{1\Q}) +  \cat(Y_{2\Q}) \leq \TC(X_\Q).$
\end{theorem}

\begin{proof}
We follow the same steps as in the proof of \thmref{thm: cat Y_1 Y_2}.  Write $Y_\Q = Y_{1\Q} \times Y_{2\Q}\times X_\Q^{n-2}$, and $f_\Q = f_{1\Q} \times f_{2\Q} \times 1_{X_\Q^{n-2}} \colon Y_\Q \to X_\Q^n$.  Then let $N(f_\Q,P_{n\Q})$ denote the pullback along $f_\Q$ of $P_{n\Q}\colon PX_\Q \to X_\Q^n$,  
and construct the same ladder of fibrations as in the proof of \thmref{thm: cat Y_1 Y_2}, thus:
$$\xymatrix{\Omega Y_\Q\ar[r]^{\Omega f_\Q} \ar[d]_{\delta} & \Omega X_\Q^n \ar[d]^{\partial} \\
 \Omega X^{n-1}_\Q \ar@{=}[r] \ar[d]_{j_\Q} &  \Omega X^{n-1}_\Q \ar[d]\\
 N(f_\Q,P_{n\Q})\ar[r] \ar[d]_{\overline{P_{n\Q}}} & PX_\Q \ar[d]^{P_{n\Q}} \\
Y_\Q \ar[r]_{f_\Q} & X_\Q^n.}$$
Notice that, here, we are making various natural identifications, including  $(X^n)_\Q = (X_\Q)^n$, $(\Omega X)_\Q = (\Omega X)_\Q$, and $(PX)_\Q = P(X_\Q)$---this latter already remarked upon at the end of the proof of \lemref{lem: local TC}.  Exactly as in the proof of \thmref{thm: cat Y_1 Y_2}, in the long exact homotopy sequence of the left-hand vertical map, we now identify the connecting homomorphism (now in rational homotopy groups) $\delta_\#\colon \pi_i(\Omega Y_\Q) \to \pi_{i}(\Omega X_\Q^{n-1})$, for $i \geq 1$, as an \emph{injective} homomorphism
$$\delta_\# = \partial_\#\circ f_{\Q\#}\colon \pi_{i+1}(Y_{1\Q}) \times \pi_{i+1}(Y_{2\Q}) \times \prod_{j=3}^{n} \pi_{i+1}(X_\Q) \to \prod_{j=1}^{n-1} \pi_{i+1}(X_\Q).$$
From the long exact sequence in rational homotopy groups, this implies that the fibre inclusion 
$$j_\Q\colon  \Omega X^{n-1}_\Q \to 
 N(f_\Q,P_{n\Q})$$
 is onto in rational homotopy groups. Now we rely upon a fact peculiar to the rational setting:  because the domain of $j_\Q$ is a (rational)  H-space, and $j_\Q$  is onto in rational homotopy groups, it follows that we have a (rational) section $\sigma \colon N(f_\Q,P_{n\Q}) \to  \Omega X^{n-1}_\Q$ of $j_\Q$.  Indeed, it is also true that $N(f_\Q,P_{n\Q})$ is a (rational) H-space: see the proof of the mapping theorem given in \cite[Th.4.11]{CLOT03}. Then it follows that we have $\overline{P_{n\Q}} = \overline{P_{n\Q}}\circ j_\Q\circ \sigma = * \circ \sigma = *\colon N(f_\Q,P_{n\Q}) \to Y_\Q$.   Consequently, from the bottom pullback square, we have
 $$\cat(Y_\Q) = \secat( \overline{P_{n\Q}}) \leq \TC_n(X_\Q).$$
 Finally, we may rewrite $\cat(Y_\Q)$ as the sum $ \cat(Y_{1\Q}) +  \cat(Y_{2\Q}) + (n-2)\,\cat(X_\Q)$, because the product equality for L-S category holds rationally \cite{F-H-L98}.  
\end{proof}

We may draw the same corollary from this proof as we did from that of  \thmref{thm: cat Y_1 Y_2}.

\begin{corollary}[to the proof of \thmref{thm: rational cat Y_1 Y_2}]\label{cor: mapping th} 
For maps $f_{j\Q} \colon Y_{j\Q} \to X_\Q$  that satisfy the hypotheses of  \thmref{thm: rational cat Y_1 Y_2}, we have
$\cat(Y_{1\Q}) \leq \secat(f_{2\Q})$. 
\end{corollary}

\begin{proof}
Use the argument of \corref{cor: cat secat} with the 
 homotopy pullback 
$$\xymatrix{ M(f_{1\Q},f_{2\Q})\ar[r]^-{\overline{f_{1\Q}}} \ar[d]_{\overline{f_{2\Q}}} & Y_{2\Q} \ar[d]^{f_{2\Q}} \\
Y_{1\Q} \ar[r]_{f_{1\Q}} & X_\Q.}$$
Note that $\overline{f_{2\Q}} \colon M(f_{1\Q},f_{2\Q}) \to
Y_{1\Q}$  factors through  $\overline{P_{2\Q}} \colon N(f_{\Q},P_{2\Q}) \to
Y_{\Q}$, where $N(f_{\Q},P_{2\Q})$ denotes the pullback of $P_{2\Q} \colon X_\Q \to X_\Q \times X_\Q$ along $f_\Q = f_{1\Q}\times f_{2\Q} \colon Y_{1\Q} \times Y_{2\Q} \to X_\Q \times X_\Q$.  In the proof of \thmref{thm: rational cat Y_1 Y_2}, we we showed $\overline{P_{2\Q}} $ to be (rationally) nulhomotopic, and hence $\overline{f_{2\Q}} $ is nulhomotopic.  The result follows.
\end{proof}

We continue with several illustrations of how \thmref{thm: rational cat Y_1 Y_2} may be applied.
We say that a fibration $p \colon E \to B$ of simply connected spaces admits a \emph{rational section} if the rationalization $p_\Q\colon E_\Q \to B_\Q$ admits a section.  In this case, we may apply   
\corref{cor: cat F B} to the rationalized fibration sequence, and conclude that $\cat(F_\Q) + \cat(B_\Q) \leq \TC(E_\Q)$.  Because of the relaxed hypotheses of \thmref{thm: rational cat Y_1 Y_2}, however, we are able to obtain the following somewhat more general result.

\begin{corollary}\label{cor: rational sectioned fibration}
Suppose that  $i\colon Y \to X$ is a map of simply connected spaces that induces an injection on rational homotopy groups, and $F \to Y \to B$ is a fibration of $Y$ by simply connected spaces that admits a rational section.  Then we have
$$\cat(F_\Q) + \cat(B_\Q) + (n-2)\,\cat(X_\Q)  \leq \TC_n(X_\Q).$$
\end{corollary}

\begin{proof}
Suppose the maps involved in the fibration of $Y$, after rationalization, are
$$\xymatrix{ F_\Q \ar[r]^{j} & Y_\Q \ar[r]_{p} & B_\Q \ar@/_1pc/[l]_{\sigma},}$$
with $p\circ\sigma = 1\colon B_\Q \to B_\Q$.  Then $i\circ j \colon F_\Q \to X_\Q$ and $i\circ \sigma \colon B_\Q \to X_\Q$ satisfy the hypotheses of \thmref{thm: rational cat Y_1 Y_2}: they both induce injections in rational homotopy groups; and it is easy to see that their image subgroups have trivial intersection in $\pi_*(X_\Q)$.  Indeed, we have a direct sum decomposition $\pi_*(Y_\Q) \cong  j_\#\big(\pi_*(F_\Q)\big) \oplus \sigma_\#\big(\pi_*(B_\Q)\big)$.  Suppose that we have $a \in \pi_*(F_\Q)$,   $b \in \pi_*(B_\Q)$ such that    $(i\circ j)_\#(a) =  (i\circ \sigma)_\#(b) \in \pi_*(X_\Q)$.  Then  $i_\#\big(  j_\#(a)- \sigma_\#(b)\big) = 0$. Hence $j_\#(a)-  \sigma_\#(b) = 0 \in \pi_*(Y_\Q)$, as $i$ is injective in rational homotopy, and thus we have $a =0 \in \pi_*(F_\Q)$ and  $b =0 \in \pi_*(B_\Q)$, since
$j_\#\big(\pi_*(F_\Q)\big) \cap \sigma_\#\big(\pi_*(B_\Q)\big)= 0 \in \pi_*(Y_\Q)$, and both $j$ and $\sigma$ are injective in rational homotopy.
The result follows.
\end{proof}

\begin{remark}
Suppose we have a (rationally) sectioned fibration $F \to E \to B$, to which we may apply \corref{cor: cat F B} or \corref{cor: rational sectioned fibration}.  Since 
 $F \to E$  induces an injection on rational homotopy groups, and since $B_\Q$ is dominated by $E_\Q$,  we have
$$\cat(F_\Q) \leq \cat(E_\Q) \leq \TC(E_\Q) \quad \textrm{and} \quad \TC(B_\Q) \leq \TC(E_\Q),$$
with the first inequality following from the mapping theorem (\thmref{introthm: mapping th}).  Thus $\cat(F_\Q)$ and $\TC(B_\Q)$ individually give lower bounds for $\TC(E_\Q)$.  It is clear, therefore, that any gain from using this theorem will come in situations in which we have a relatively large value for the sum $\cat(F_\Q) + \cat(B_\Q)$, when compared with either $\cat(E_\Q)$ or $\TC(B_\Q)$.
\end{remark}

In the next example, and at several places in the sequel, we make use of minimal models.  These are a basic tool of rational homotopy theory that allows one to work in an entirely algebraic setting. 
We refer to \cite[Sec.12 \emph{et seq.}]{F-H-T} and \cite[Ch.2]{F-O-T08} for comprehensive treatments of minimal models.  Here, we simply recall that the minimal model of a space $X$ is a differential graded (DG) algebra, $\land(V;d)$, where $\land V$ denotes the free graded commutative algebra generated by the graded vector space $V$, and $d$ denotes a decomposable differential, that is, we have $d \colon V \to \land^{\geq 2} V$.    The minimal model encodes all of the rational homotopy information of $X$, in a certain technical sense.  Some rational homotopy data are readily extracted from it.  For example, on passing to cohomology we obtain $H(\land(V;d)) \cong H^*(X;\Q)$, and as graded vector spaces we have isomorphisms $\Hom(V,\Q) \cong \pi_*(X)\otimes\Q$.    As another example, Whitehead product structure in  $\pi_*(X)\otimes\Q$ corresponds to the quadratic part of the differential $d$ (see \cite[Sec.13 (e)]{F-H-T} for details).

A number of authors have used algebraic or minimal model versions of $\TC(-)$ to study rationalized topological complexity (e.g.~\cite{F-G-K-V06, JMP12}).  Whilst we use  minimal models  in some of our examples and applications,  we do not make use of an algebraic, or minimal model, version of $\TC_n(-)$ as such. 
The following example illustrates how, in certain instances, our 
classical approach to rational $\TC(-)$ might serve to replace the more 
technical minimal model approach.  In the example, $\MTC(X)$ is the so-called \emph{module topological complexity},  an algebraic invariant defined in terms of the minimal model.  This invariant can sometimes give a greater lower bound for (rational) $\TC(X)$ than  the rational zero-divisors lower bound which uses the rational cohomology algebra.   

\begin{example}\label{ex:MTC example}
We analyze an example from \cite{F-G-K-V06}.  Let $X = S^3 \vee S^3 \cup_a e^8 \cup_b e^8$, with the $8$-cells attached by Whitehead products $a = [\iota_1, [\iota_1, \iota_2]]$ and $b = [\iota_2, [\iota_1, \iota_2]]$.  In \cite[Ex.6.5]{F-G-K-V06}, it is shown that  $\MTC(X) = 3$, which gives a better lower bound for (rational) $\TC(X)$ than does the rational zero-divisors lower bound (which is $2$, here).      We will use \thmref{thm: rational cat Y_1 Y_2} to match this lower bound given by $\MTC(X)$.

We construct maps $f_j \colon Y_j \to X$ with the salient properties.  We may describe (the first few terms of) the minimal model of $X$ as follows:
$$\land( u_3, v_3, w_5, x_{10}, \ldots)$$
with subscripts denoting degrees, and differentials $d(u) = 0 = d(v)$, $d(w) = uv$, $d(x) = uvw$, and so-on.  Notice that the cycles $uw$ and $vw$ generate cohomology in degree $8$, and that $X$ is not formal.   Now we have an obvious projection of the minimal model
$$\land( u_3, v_3, w_5, x_{10}, \ldots) \to \land(u, w; \bar{d} = 0)$$
that corresponds to a map $f_1\colon Y_1 \to X$ injective in rational homotopy.  Indeed, $Y_1 \simeq_\Q S^3 \times S^5$.   On the other hand, we also have a projection
$$\land( u_3, v_3, w_5, x_{10}, \ldots) \to \land(v, x, \ldots).$$
Here, the quotient differential will be zero for the first terms, but generally not zero.  This likewise corresponds to a map $f_2\colon Y_2 \to X$ injective in rational homotopy.  Note that  $Y_2$ fibres rationally over $S^3$ with fibre  $X^{[9]}$, the $9$-connective cover of $X$.  The images  $(f_j)_\#\big(\pi_*(Y_{j\Q}) \big)$ in $\pi_*(X_{\Q})$ are distinct.   Indeed, the only degree in which they could possibly intersect non-trivially is in degree $3$, and here the images are distinct by choice. 
Applying \thmref{thm: rational cat Y_1 Y_2},  we obtain that
$$\TC(X_\Q) \geq \cat(Y_\Q) = \cat(Y_{1\Q}) + \cat(Y_{2\Q}) \geq  2 + 1 = 3,$$
since $Y_2$ is not rationally contractible.  If it were possible to show that $Y_2$ is not a rational co-H space, then we could improve this lower bound to  $\TC(X_\Q) \geq 4$.  As it is,  we have already matched the lower bound on $\TC(X)$ obtained in \cite{F-G-K-V06} using $\MTC(X)$ .
\end{example}

Connective covers  provide one source of spaces $Y_j$ for \thmref{thm: rational cat Y_1 Y_2}.  Recall our notation for the connective cover: $X^{[N]}$ is $N$-connected.  Now  \emph{rationally}, from the mapping theorem (\thmref{introthm: mapping th}), we have $\cat(X^{[N]}_\Q) \leq \cat(X_\Q)$ for any $N \geq 2$.   We illustrate the use of connective covers in this setting  with a proof of a rational version of the main result of \cite{GLO13}.

\begin{corollary}\label{cor: rational TC = 1}
Let $X$ be a  simply connected space with $\TC_n(X_\Q) = n-1$, some $n \geq 2$.  Then $X \simeq_\Q S^{2r+1}$ for some $r \geq 1$.
\end{corollary}

\begin{proof}
First, $X$ must have at least one non-zero odd-degree rational homotopy group.  For otherwise,  $X$ is rationally equivalent to a product of even-dimensional Eilenberg-Mac Lane spaces, and we have $\infty = \cat(X_\Q) = \TC(X_\Q)$.  So suppose that $\pi_{2r+1}(X_\Q) \not= 0$, for some $r \geq 1$, and that all odd-degree rational homotopy groups of $X$ below this degree are zero.  Then we may fibre the minimal model of $X^{[2r]}$ as
$$\xymatrix{ \land(v) \ar[r] & (\land(v) \otimes \land(W), d) \ar@/_1pc/[l] \ar[r] & (\land(W), \bar{d}) },$$
with $v$ a  generator in degree $2r+1$.  Topologically, this corresponds to a fibration with section
$$\xymatrix{ F_\Q \ar[r]^{j} & X^{[2r]}_\Q \ar[r]_{p} & S^{2r+1}_\Q \ar@/_1pc/[l]_{\sigma}.}$$
Note that, since $(n-1) \cat(X_\Q) \leq \TC_n(X_\Q)$, our hypothesis implies that $\cat(X_\Q) = 1$, so  that $X$ is a rational co-H space.   Note also that we have
$\cat(S^{2r+1}_\Q) = 1$.
Now apply \corref{cor: rational sectioned fibration}, and obtain that $\cat(F_\Q) + 1 + (n-2) \leq \TC(X_\Q) = n-1$, whence $\cat(F_\Q) = 0$, and so we have $X^{[2r]} \simeq_\Q S^{2r+1}$.

There are two possibilities that remain.  Either we have all rational homotopy groups of degree below $2r+1$ zero, in which case we have $X \simeq_\Q S^{2r+1}$. Or, if there are  non-zero rational homotopy groups of degree less than $2r+1$,  the minimal model of $X$ must be of the form $\land(u_1, \dots, u_k, v; d)$, with each $u_i$ of even degree (recall that we assumed $2r+1$ was the lowest odd degree in which $X$ had a non-zero rational homotopy group).  In this latter case, the only possible non-zero differential is $d(v) = P$, some polynomial in degree $2r+2$ in the $u_i$.  A straightforward argument shows that, if $k \geq 2$, then $X_\Q$ has a non-zero cup-product in cohomology---indeed, must have infinite cup-length---and hence $X$ cannot be a rational co-H space.    On the other hand, if  $k=1$, then  $H^*(X;\Q) \cong \Q[u_1]/(u_1^n)$, a truncated polynomial algebra on a single even-degree generator.  For such a space, the zero-divisors lower bound implies that $\TC(X_\Q) \geq 2$, and so this case may also be eliminated.  The only remaining possibility, then, is that we have $X \simeq_\Q S^{2r+1}$.
\end{proof}

\begin{remark}
An interesting aspect of the above proof is that it makes very little use of the usual zero-divisors lower bounds.  
By contrast, the main result of \cite{GLO13} was proved by gleaning cohomological data and repeatedly appealing to the zero-divisors lower bound.
\end{remark}

We give another consequence of \corref{cor: rational sectioned fibration}, of a rather general nature.  

\begin{corollary}\label{cor: zero bracket}
Let $X$ be a simply connected, hyperbolic  space.   Suppose that we have linearly independent elements $a \in \pi_{2r+1}(X_\Q)$ and $b \in \pi_{2s+1}(X_\Q)$,
$r \leq s$, with zero Whitehead product: $[a,b] = 0 \in \pi_{2r + 2s+1}(X_\Q)$. Then we have $\TC_n(X_\Q) \geq 2n-1$, for each $n \geq 2$.
\end{corollary}

\begin{proof}
We argue with minimal models.  First, write the minimal model of $X$ as
$$\land(x_1, \ldots, x_k, a, x_{k+1}, \ldots, x_l, b, w_1, \ldots; d),$$
with generators in non-decreasing degree order and---by abuse of notation---$a$ and $b$ the odd-degree generators that correspond to  the linearly independent rational homotopy elements in the hypotheses.  Then the ideal generated by $\{x_1, \ldots, x_k, x_{k+1},$ $\ldots, x_l\}$ (without $a$) is $d$-stable, and we may project onto the quotient by this ideal
$$\land(x_1, \ldots, x_k, a, x_{k+1}, \ldots, x_l, b, w_1, \ldots; d) \to \land(a,  b, w_1, \ldots; \bar{d}).$$
This corresponds to a map $Y_\Q \to X_\Q$ that is injective in rational homotopy groups.  Next, we may fibre the model $ \land(a,  b, w_1, \ldots; \bar{d})$ of $Y$ as
$$\xymatrix{\land(a,b) \ar[r] & (\land(a,b) \otimes \land(W), \bar{d}) \ar@/_1pc/[l] \ar[r] & (\land(W), \bar{d}' )}$$
which corresponds to a fibration with section
$$\xymatrix{ F_\Q \ar[r]^{j} & Y_\Q \ar[r]_-{p} & S^{2r+1}_\Q \times S^{2s+1}_\Q \ar@/_1pc/[l]_{\sigma}.}$$
Note that we obtain the section (retraction in minimal models) because of the assumption about the Whitehead product vanishing: in minimal models, this means  that the term $ab$ does not occur in any differential (in $d$, and hence in $\bar{d}'$).  Because of our assumption that $X$ is hyperbolic, it follows that $Y$ is hyperbolic and, in particular, $F_\Q$ cannot be contractible.  Now apply \corref{cor: rational sectioned fibration}, to obtain that $\TC_n(X_\Q) \geq \cat(F_\Q) + 2 + (n-2) \cat(X_\Q)$.  Since $X_\Q$ does not have a free homotopy Lie algebra, we have that $\cat(X_\Q) \geq 2$.  Hence we have   $\TC_n(X_\Q)  \geq \cat(F_\Q) + 2n -2 \geq 2n-1$, since $\cat(F_\Q) \geq 1$. 
\end{proof}

This last result indicates an intriguing  connection between (vanishing of) products in  the homotopy Lie algebra and $\TC(-)$.  We pursue this connection in more detail  in the last section.

\section{From Rational Results to Integral Results}\label{sec: rational to integral}

In this section, we illustrate how our rational lower bounds may lead to a determination of ordinary (higher) $\TC(-)$.  We first recall the rational version of \corref{cor: TC X v Y}.

\begin{corollary}\label{cor: rational wedge}
Let $X$ and $Y$ be simply connected, of finite rational type.  Then for $n \geq 2$, we have
$$\cat(X_\Q) + \cat(Y_\Q) + (n-2) \max\{ \cat(X_\Q), \cat(Y_\Q) \} \leq \TC_n(X_\Q \vee Y_\Q).$$
\end{corollary}

\begin{proof}
Either apply \corref{cor: TC X v Y} to the wedge $X_\Q \vee Y_\Q$, or apply \thmref{thm: rational cat Y_1 Y_2} directly to the inclusions  $i_1\colon X \to X \vee Y$ and $i_2\colon Y \to X \vee Y$.
\end{proof}

We may use this result to identify (integral)  topological complexity of  wedges of certain kinds of spaces.  Let $X$ be a simply connected space of finite type.  We say that $X$ is an \emph{oddly generated space} if its rational homotopy groups are concentrated in finitely many odd degrees.  In this case, $X$ has minimal model of the form $\land(v_1, \ldots,$ $v_k; d)$, with each $|v_i|$ odd---not necessarily of distinct degrees---and then we say that $X$ \emph{has rank $k$}.
Any iterated sequence of principal bundles starting from a 
product of odd spheres has minimal model of this form.  Lie groups themselves 
have such models with zero differential.

For the following result, recall that---supposing we have $X$ and $Y$ oddly generated and both of rank $k$--- we have a  general upper bound of 
$$\TC_n(X \vee Y) \leq \cat\big((X\vee Y)^n\big) \leq n\;\mathrm{max}\{\cat(X), \cat(Y)\}.$$

\begin{theorem}\label{thm: wedge of odd}
Suppose $X$ and $Y$ are simply connected, oddly generated spaces both of  rank $k$.    If $X$ and $Y$ satisfy $\cat(X)=\cat(Y) =k$, then we have $\TC_n(X \vee Y) = nk$.
\end{theorem}

\begin{proof}
For an oddly generated space $X$ of rank $k$, we have $\cat(X_\Q) = k$.  Applying \corref{cor: rational wedge}, we have
$$nk =  \cat(X_\Q) + \cat(Y_\Q) + (n-2) \max\{ \cat(X_\Q), \cat(Y_\Q) \} \leq \TC_n(X_\Q \vee Y_\Q).$$
But  $\TC_n(X_\Q \vee Y_\Q) \leq \TC_n(X \vee Y)\leq  \cat\big((X\vee Y)^n\big) \leq n \cat(X \vee Y) = nk$.
\end{proof}

To apply the result, we need examples in which $X$ and $Y$ are both oddly generated spaces and of the same rank as each other.  For instance, we may take $X$ or $Y$ to be one of: a product of $k$ odd-dimensional spheres; or $U(k)$;  or $SU(k+1)$ (see \cite{Sin75} or \cite[Prop.9.5, Th.9.47]{CLOT03}).
It is important to realise that, although we are starting with information about the rational homotopy of these spaces, our conclusion is one about the ordinary $\TC_n(-)$ of an (ordinary) space.  

It is not too hard to construct further examples of oddly generated spaces to which \thmref{thm: wedge of odd} may be applied, once again resulting in an exact determination of $\TC_n(-)$.

\begin{example}\label{ex: TC(XvX) = 6}
We construct a smooth manifold $X$ that is oddly generated with $\rk(X) = \cat(X) = 3$.  Take $p\colon E \to S^6$ to be the (unit) sphere bundle of the tangent bundle over $S^6$.  This is an $S^5$-bundle over $S^6$.  Now let $f\colon S^3 \times S^3 \to S^6$ be a (smooth)  map of degree $1$, and form the pullback
$$\xymatrix{X \ar[r] \ar[d]_{p'} & E \ar[d]^{p}\\
S^3 \times S^3 \ar[r]_-{f} & S^6,}$$
of $p$ along $f$.  This results in our space $X$, an $S^5$-bundle over $S^3 \times S^3$.  One sees that the minimal model of $X$ is $\land(v_1, v_2, v_3;d)$, with $|v_1| = |v_2| = 3$ and $|v_3| = 5$, with differential given by $d(v_3) = v_1 v_2$ (this kind of example is discussed in \cite{F-O-T08}).  Using rational category (in fact, the rational Toomer invariant) as a lower bound, and the usual ``dimension over connectivity" upper bound \cite[Th.1.49]{CLOT03}, we find that
$$3 = \eQ(X) \leq \cat(X_\Q) \leq \cat(X) \leq \frac{\dim(X)}{3} = \frac{11}{3} < 4,$$
and thus we have $\cat(X) = 3$.  Then we have, for example,  $\TC(X \vee X) = 6 = \cat\big( (X \vee X) \times (X \vee X) \big)$ from \thmref{thm: wedge of odd}.
\end{example}

\section{Consequences for the Avramov-F{\'e}lix Conjecture}\label{sec:A-F}

Recall the following conjecture from the introduction:

\begin{conjecture}[Avramov-F{\'e}lix]
If $X$ is a simply connected, hyperbolic finite complex, then $\pi_*(\Omega X) \otimes \Q$ contains a free Lie algebra on two generators.
\end{conjecture}

This has been established in some cases, including the case in which $X$ has $\cat(X_\Q) = 2$ (\cite{F-H-T84}, although some sub-cases are left open),  the case in which $\pi_*(\Omega X) \otimes \Q$ has depth $1$ (\cite{Bo-Ja89}, which covers the sub-cases left open for $\cat(X_\Q) = 2$), and the case in which $X$ is a Poincar{\'e} duality space, with $H^*(X;\Q)$ evenly graded and generated by at most three generators (implied by the results of \cite{Av82}).  Generally speaking, though, it remains open.

Now for a given value of $\cat(X_\Q)$, the general inequalities $\cat(X^{n-1}) \leq \TC_n(X)$ $\leq \cat(X^n)$, for each $n \geq 2$, imply that $\TC_n(X_\Q)$ is lowest when we have equality $\TC_n(X_\Q) = (n-1)\cat(X_\Q)$.  Said differently, amongst spaces with $\cat(X_\Q) = k$, those with $\TC_n(X_\Q) = (n-1)k$ for some $n \geq 2$ satisfy a particularly strong constraint, such that one might hope to prove theorems about them.   An extreme case of this point of view is given in \corref{cor: rational TC = 1}, and its integral counterpart in  \cite{GLO13}.  Here, we make some progress on the Avramov-F{\'e}lix (A-F) conjecture for spaces that satisfy  $\TC_n(X_\Q) = (n-1)\cat(X_\Q)$. 

A rational co-H-space, that is, a simply connected space $X$ with $\cat(X_\Q) = 1$, has rational homotopy Lie algebra $\pi_*(\Omega X)\otimes \Q$ a free graded Lie algebra, and thus automatically satisfies the A-F conjecture (assuming it is rationally hyperbolic).  Also, for any connective cover $X^{[N]}$ of a simply connected, finite-type space $X$, we have an inclusion of rational homotopy Lie algebras $\pi_*(\Omega X^{[N]})\otimes \Q \subseteq \pi_*(\Omega X)\otimes \Q$.    Furthermore, by \thmref{introthm: mapping th}, we know that $\cat(X^{[N]}_\Q) \leq \cat(X_\Q)$ for any cover, however there is no guarantee that $\cat(-)$ will actually decrease upon passing to a cover. 
Our basic strategy here is to consider connective covers, attempting to reduce $\cat(X^{[N]}_\Q)$ to a value for which the A-F conjecture is known to hold.

This strategy prompts the following questions.  

\begin{question}\label{que: co-H cover}
When does a simply connected, hyperbolic finite complex have a connective cover that is a rational co-H-space? More generally, when does a simply connected, hyperbolic finite complex have a connective cover of strictly lower rational category?
\end{question}

\begin{example} In which we observe  that not every space has a connective cover that is a rational co-H-space, and not every space has a connective cover of strictly lower rational category.  For take $X = (S^3 \vee S^3) \times (S^3 \vee S^3)$.  Then for each $N$, we have $X^{[N]} = (S^3 \vee S^3)^{[N]} \times (S^3 \vee S^3)^{[N]}$.   It is easy to see that here we have $\cat(X_\Q) = \cat(X^{[N]}_\Q) = 2$ for each $N$.   We may adapt this example to one in which $\cat(X^{[N]}_\Q)$ decreases from $\cat(X_\Q)$ to any intermediate value between $\cat(X_\Q)$ and $1$, and stabilizes at that value.  For instance, take $X =  (S^3 \vee S^3) \times (S^3 \vee S^3) \times S^5 \times S^7$.  Then we have $\cat(X_\Q) = 4$,  $\cat(X^{[5]}_\Q) = 3$, and $\cat(X^{[N]}_\Q) = 2$ for $N \geq 7$.
\end{example}

\noindent{}We give a partial response to \queref{que: co-H cover} in our next result. We will then specialize to  small values of  $\cat(X_\Q)$ ($\leq 3$) and establish  some cases of the A-F conjecture.

In the following results, note that  $\pi_\mathrm{odd}(X) \otimes \Q$ must be non-zero, otherwise $X$ would have the rational homotopy type of a product of even-dimensional Eilenberg-Mac Lane spaces, and could not be finite.

\begin{theorem}\label{thm: cat cover goes down} 
Let $X$ be a simply connected, hyperbolic finite complex.  Suppose that we have $\TC_n(X_\Q) =  (n-1)\cat(X_\Q)$ for some $n \geq 2$.
\begin{itemize}
\item[(1)] Suppose $2r+1$ is the lowest odd degree in which  $\pi_*(X) \otimes \Q$ is non-zero.  Then we have
$$\cat(X_\Q^{[2r+1]}) \leq \cat(X_\Q) - 1.$$
\item[(2)]  Suppose that  there exist linearly independent elements $a \in \pi_{2p}(\Omega X)\otimes \Q$, $b \in \pi_{2p}(\Omega X)\otimes \Q$, $p\leq q$, with Samelson product $\langle a, b\rangle  = 0 \in \pi_{2p+2q}(\Omega X)\otimes \Q$.  Then we have 
$$\cat(X_\Q^{[2q+1]}) \leq \cat(X_\Q) - 2.$$ 
\end{itemize}
\end{theorem} 

\begin{proof}
(1)     Choose any essential map $\alpha\colon S^{2r+1} \to X_\Q$.  Because odd-dimensional spheres are rational Eilenberg-Mac Lane spaces, $\alpha$ is injective in (all) rational homotopy groups.  Then the maps $\alpha\colon Y_1 = S^{2r+1} \to X_\Q$ and $i\colon X^{[2r+1]}_\Q \to X_\Q$ satisfy the hypotheses of   \thmref{thm: rational cat Y_1 Y_2}, yielding an inequality
$$ 1 + \cat(X^{[2r+1]} _\Q) + (n-2) \cat(X_\Q) \leq \TC_n(X_\Q) =  (n-1)\cat(X_\Q).$$
It follows that we must have $\cat(X^{[2r+1]} _\Q) \leq \cat(X_\Q) - 1$, as asserted.
 
(2) Under the standard identifications of $\pi_*(\Omega X)$  with $\pi_{*+1}(X)$, and Samelson product in $\pi_*(\Omega X)$ with Whitehead product in $\pi_{*}(X)$, the assumption may be phrased as follows:  For  linearly independent elements $\alpha \in \pi_{2p+1}(X_\Q)$, $\beta \in \pi_{2q+1}(X_\Q)$,  we have  $[\alpha, \beta] = 0 \in \pi_{2p+2q+1}(X_\Q)$, where the bracket denotes their Whitehead product.  Because their Whitehead product is zero, we have a map $f_1\colon Y_1 = S^{2p+1} \times S^{2q+1} \to X_\Q$ that restricts to $(\alpha\mid \beta)\colon S^{2p+1} \vee S^{2q+1} \to X_\Q$ on the wedge.  Once again, because odd-dimensional spheres are rational Eilenberg-Mac Lane spaces, $f_1$ is injective in (all) rational homotopy groups.   Then the  maps $f_1$ and  $f_2\colon Y_2 = X^{[2q+1]}_\Q \to X_\Q$ satisfy the hypotheses of   \thmref{thm: rational cat Y_1 Y_2}, and we obtain an inequality
$$2 + \cat(X^{[2q+1]} _\Q) + (n-2) \cat(X_\Q) \leq \TC_n(X_\Q) =  (n-1)\cat(X_\Q),$$
whence we have $\cat(X^{[2q+1]} _\Q) \leq \cat(X_\Q) - 2$.
\end{proof}

For a  graded Lie algebra $L$, we say that \emph{$L$ has no zero brackets} if, whenever $x,y \in L$ are linearly independent elements, we have $[x,y] \not=0 \in L$.

Recall again the general inequality $(n-1) \cat(X_\Q) \leq \TC_n(X_\Q)$, for each $n\geq 2$.  If we  assume that $\TC_n(X_\Q) \leq 2n-3$, for some $n$, then it follows that $\cat(X_\Q) =1$, and $X$ is a rational co-H-space.  Since we are assuming $X$ is hyperbolic, this implies that $X$ has same rational homotopy type as a wedge of at least two spheres, and hence satisfies the A-F conjecture.  The next step, therefore, is the case in which  $\TC_n(X_\Q) \leq 2n-2$, for some $n\geq 2$.  The interest here is in the case in which $\cat(X_\Q) = 2$ and $\TC_n(X_\Q) = 2n-2$ for some $n$.  Any space with $\cat(X_\Q) = 2$ is already known to satisfy the A-F conjecture \cite{F-H-T84, Bo-Ja89}.  Here, we will obtain a stronger conclusion using the extra constraint on $\TC_n(X_\Q)$.

\begin{corollary}\label{cor: TC 2 A-F}
Let $X$ be a simply connected, hyperbolic finite complex, and suppose that 
$\TC_n(X_\Q) \leq 2n-2$, for some $n\geq 2$.  Then $\pi_\mathrm{even}(\Omega X) \otimes \Q$ has no zero brackets.  Furthermore, if $\pi_\mathrm{2r}(\Omega X) \otimes \Q$ is the lowest-degree non-zero part of $\pi_\mathrm{even}(\Omega X) \otimes \Q$, then the connective cover $X^{[2r+1]}_\Q$ is a (rational) co-H-space.  In particular, $\pi_*(\Omega X) \otimes \Q$ contains as a sub-Lie algebra the infinite-dimensional free Lie algebra  $\pi_*(\Omega X^{[2r+1]}) \otimes \Q$.
\end{corollary}

\begin{proof}
The hypothesis $\TC_n(X_\Q) \leq 2n-2$, for some $n\geq 2$ entails $\cat(X_\Q) \leq 2$.   If $\cat(X_\Q) = 1$, then $\pi_*(\Omega X) \otimes \Q$  is a free Lie algebra and we are done.  So, assume that $\cat(X_\Q) = 2$.   First, suppose that we have two linearly independent, even-degree elements in  $\pi_*(\Omega X) \otimes \Q$  whose bracket is zero.  Then part (2) of \thmref{thm: cat cover goes down} contradicts the fact that each connective cover $X_\Q^{[N]}$ is non-contractible.   Hence,  $\pi_\mathrm{even}(\Omega X) \otimes \Q$ has no zero brackets.   
The remaining assertion follows directly from  part (1) of \thmref{thm: cat cover goes down}. 
\end{proof}

\begin{remark}
The condition that a graded Lie algebra have no zero brackets is, in general,  strictly weaker than the condition that the Lie algebra be free.  If a graded Lie algebra is free, then it has no zero brackets.  The example below illustrates that the converse need not be true.  Note, however, that our example is not the rational homotopy Lie algebra of a hyperbolic space, as it does not have the (exponentially) increasing ranks displayed by such.  Clearly, no zero brackets entails a rich bracket structure.  It would be interesting to understand more fully the relationship, if any, between no zero brackets and free-ness, in the context of hyperbolic rational homotopy Lie algebras. 
\end{remark}

\begin{example}
Let $L$ be the evenly graded vector space with $L_{2i} = \langle x_i \rangle$, a one-dimensional vector space with basis element $x_i$, for each $i = 1, 2, \ldots$.  Define a bracket structure on $L$ by setting $[x_i, x_j] = 2(j-i) x_{i+j}$, for $i < j$.  Then $L$ is a graded Lie algebra with no zero brackets.  Evidently, $L$ is not a free graded Lie algebra (nor does it contain a sub-Lie algebra that is free on two generators). This example is based on  the Witt algebra; we thank Simon Wadsley for pointing it out to us (via MathOverflow).
\end{example}

We continue to probe the A-F conjecture, relaxing the constraint on $\TC_n(-)$.   For $\TC_n(X_\Q)$ in the range $2n-2 \leq \TC_n(X_\Q) \leq 3n-4$, for any $n$, we are still constrained to $\cat(X_\Q) \leq 2$.  However, in the range $2n-1 \leq \TC_n(X_\Q) \leq 3n-4$, we do not have $\TC_n(X_\Q)$ as low as possible, given the value of $\cat(X_\Q)$.  The next step, then, is to consider $\TC_n(X_\Q) \leq 3n-3$, which now allows for $\cat(X_\Q) = 3$.  We will ``bootstrap," using the affirmative solution to the A-F conjecture for spaces of rational category $2$, and establish the A-F conjecture here.  As with the previous result, we will actually obtain somewhat finer information.

\begin{corollary}\label{cor: TC 3 A-F}
Let $X$ be a simply connected, hyperbolic finite complex with $\cat(X_\Q)$ $= 3$ and 
$\TC_n(X_\Q) = 3n-3$, for some $n\geq 2$.  If there exist linearly independent elements $a, b  \in \pi_\mathrm{even}(\Omega X) \otimes \Q$ with $\langle a, b\rangle  = 0 $, and $|a| \leq |b| = 2q$,  then $X^{[2q+1]}_\Q$ is a rational co-H-space.  Independently of whether or not this is the case,  we have $\cat(X^{[2r+1]}_\Q) \leq 2$, where $2r$ is the lowest-degree non-zero part of $\pi_\mathrm{even}(\Omega X) \otimes \Q$.  In all cases, $X$ satisfies the A-F conjecture.
\end{corollary}

\begin{proof}
This follows directly from part (2) of \thmref{thm: cat cover goes down}.  In the second case here, we rely upon \cite{F-H-T84, Bo-Ja89} to conclude the A-F conjecture for $X^{[N]} _\Q$, and thus for $X_\Q$.
\end{proof}

Returning briefly to the situation in which $\cat(X_\Q) = 2$, and  $2n-1 \leq \TC_n(X_\Q) \leq 3n-4$ is not as low as possible, there is one more consequence to be gleaned from 
\thmref{thm: rational cat Y_1 Y_2}.  This concerns \queref{que: co-H cover}. 

\begin{corollary}\label{cor: cat 2 TC 3}
Let $X$ be a simply connected, hyperbolic finite complex with $\cat(X_\Q)$ $= 2$ and 
$\TC_n(X_\Q) = 2n-1$, for some $n\geq 2$.  If there exist linearly independent elements $a, b  \in \pi_\mathrm{even}(\Omega X) \otimes \Q$ with $\langle a, b\rangle  = 0 $, and $|a| \leq |b| = 2q$,  then $X^{[2q+1]}_\Q$ is a rational co-H-space.
\end{corollary}

\begin{proof}
This  follows from the same argument as was used to show part (2) of \thmref{thm: cat cover goes down}. 
\end{proof}

It is irresistible to imagine somehow using \thmref{thm: cat cover goes down} inductively, so as to address \queref{que: co-H cover} for spaces that satisfy $\TC_n(X_\Q) =  (n-1)\cat(X_\Q)$.  Unfortunately, we are not able to do so at present.
Another question prompted by our results here is the following.   In principle, the hypothesis that $\TC_n(X_\Q) =  (n-1)\cat(X_\Q)$  gives a different hypothesis on $X$ for each  $n \geq 2$.  However, there is some evidence to suggest they are not separate.  

\begin{question}
Does $\TC_n(X) =  \cat(X^{n-1})$ for some $n\geq 2$ imply that $\TC_n(X) =  \cat(X^{n-1})$ for all $n \geq 2$?
\end{question}

\providecommand{\bysame}{\leavevmode\hbox to3em{\hrulefill}\thinspace}
\providecommand{\MR}{\relax\ifhmode\unskip\space\fi MR }
\providecommand{\MRhref}[2]{%
  \href{http://www.ams.org/mathscinet-getitem?mr=#1}{#2}
}
\providecommand{\href}[2]{#2}


\end{document}